\documentclass[12 pt]{amsart}

\usepackage{amssymb,amscd,stmaryrd}
\usepackage[left=3.3cm,right=3.3cm]{geometry}
\usepackage[mathcal]{eucal}
\usepackage{array,float}

\usepackage{xy}
\input xy
\xyoption{all}
\usepackage{pdflscape}
\usepackage{hyperref}
\usepackage[vcentermath]{youngtab}
\usepackage[draft]{graphicx}

\numberwithin{equation}{section}

\newtheorem{thm}{Theorem}[section]
\newtheorem{proposition}[thm]{Proposition}

\newtheorem{corollary}[thm]{Corollary}
\newtheorem{lemma}[thm]{Lemma}

\theoremstyle{definition}
\newtheorem{definition}[thm]{Definition}
\newtheorem*{remark*}{Remark}
\newtheorem{remark}[thm]{Remark}

\newcommand{\End}{\text{End}}

\newcommand{\ind}{\mathrm{Ind}}

\title[On irreducible representations of nilpotent groups]{On monomial representations of finitely generated nilpotent groups}
\date{\today }
\author{E. K. Narayanan}
\address{ Department of Mathematics,
    Indian Institute of Science,
    Bangalore 560012, India }

\email{naru@math.iisc.ernet.in}
\author{Pooja Singla}

\email{pooja@math.iisc.ernet.in}


\begin{document}
\maketitle
\begin{abstract}
A result of D. Segal states that every complex irreducible representation of a finitely generated nilpotent group $G$ is monomial if and only if $G$ is abelian-by-finite. A conjecture of A. N. Parshin, recently proved affirmatively by I.V. Beloshapka and S. O. Gorchinskii (2016), characterizes the monomial irreducible representations of finitely generated nilpotent groups.
	This article gives a slightly shorter proof of the conjecture combining the ideas
	of I. D. Brown and P. C. Kutzko.
	We also characterize finite dimensional irreducible representations of two step nilpotent groups and also
provide a full description of the finite dimensional representations of two step
	groups whose center has rank one.
\end{abstract}

\maketitle

\bigskip
\section{Introduction}
It is well known that every finite-dimensional irreducible representation
of a nilpotent group over the field of complex numbers is monomial, that is
induced from a one dimensional representation of some subgroup. Hall~\cite{MR0110750}, proved that every complex irreducible representation of a finitely generated nilpotent group $G$ is finite dimensional if and only if $G$ is abelian-by-finite, i.e., $G$ contains an abelian normal subgroup $N$ such that $G/N$ is finite. Therefore, in general a
finitely generated nilpotent group has infinite-dimensional irreducible
representations. As a first step towards an understanding of representations of finitely generated nilpotent groups, one would like to know whether all irreducible representations are necessarily monomial.  Segal~\cite{MR0444759},  proved that every complex irreducible representation of $G$ is monomial if and only if $G$ is abelian-by-finite. Therefore the question of characterizing monomial irreducible representations of a finitely generated nilpotent groups arises naturally. The characterization that works efficiently comes from the parallel questions for the unitary irreducible representations of nilpotent Lie groups. For the motivation, we recall these results first.

For the case of simply connected nilpotent Lie groups and complex unitary irreducible representations the analogous result appeared in 1962 in the classical work of Kirillov~\cite{Kirillov1962} (see also Dixmier~\cite{Dixmier3, Dixmier4}), where he developed the famous orbit method to construct irreducible representations of simply connected nilpotent Lie groups. It turns out that every unitary irreducible representation of a simply connected
connected nilpotent Lie group is induced from a unitary character of a subgroup.

On the other hand, the complex irreducible unitary representations of finitely generated discrete nilpotent groups are not necessarily monomial. For example, I.D.Brown~\cite{MR0352324} constructed an irreducible unitary representation of discrete Heisenberg group  that is not monomial. Brown~\cite{MR0352324} also proved that a unitary irreducible representations of a discrete nilpotent group is monomial if and only if it has finite weight. Recall that a representation $\rho$ of $G$ is said to have finite weight if there exists a subgroup $H$ of $G$ and a character $\chi$ of $H$ such that space of $H$-linear maps, also called the space of intertwining operators, $\mathrm{Hom}_H(\rho|_{H}, \chi)$ is non-zero and finite dimensional.

Brown's result~\cite{MR0352324} motivates the question whether a similar characterization holds for irreducible algebraic representations of finitely generated discrete nilpotent groups. That is, if $G$ is such a group and $\pi$ is an irreducible algebraic representation of $G,$ then is it true that $\pi$ is equivalent to $\ind_{H}^G(\chi)$ for some character $\chi$ of a subgroup $H$ if and only if $\pi$ has
finite weight.
 A. N. Parshin~\cite{ParshinICM}, during plenary lecture at ICM 2010, conjectured that Brown's characterization of monomial representations holds for all complex irreducible representations of finitely generated nilpotent groups (see also \cite{Arnal-Parshin}). In this case by induction of a representation, we mean finite induction (See Definition~\ref{defn:induced}). This conjecture was proved by Beloshapka and Gorchinksii
(see \cite{Beloshapka-Gorchinskiy}).

 In this article our aim is to present a proof of this conjecture following the ideas of Brown ~\cite{MR0352324}. An earlier version of this article (which was circulated among some) which claimed to have settled Parhin's conjecture contained a gap in the proof. This was pointed out to us by Prof. Parshin (see Proposition~\ref{reduction-to-normalizer}). Beloshapka and Gorchinksii, working on the same problem independently, fixed this gap, thus proving the conjecture in full (see \cite{Beloshapka-Gorchinskiy}). Although, both these proofs were modeled on the proof by Brown~\cite{MR0352324} for unitary representations, we feel that our usage of Kutzko's results (see Theorem~\ref{thm:kutzko})  makes the proof of the main theorem shorter. The results of Sections~\ref{Sec: rank-one} and \ref{Sec: general}, we believe are new and are of independent interest. The following is the main result of this paper. We build on the ideas of this article to prove parallel results for finitely generated supersolvable groups in \cite{supersolvable}.  Unexpectedly, the proofs get much more involved for the supersolvable case as compared to the nilpotent one. It will be interesting to characterize the class of infinite polycyclic groups for which Brown's characterization holds.

\begin{thm}\label{main-theorem}

Let $G$ be a finitely generated nilpotent group. An
irreducible countable dimensional representation $\pi$ of $G$ is monomial if and only if it has
finite weight.
\end{thm}
We describe the strategy of the proof here.
First of all we obtain an important sufficient condition for a subgroup $H$ of a nilpotent group $G$ and its character $\chi$ such that $\mathrm{End}_G(\ind_H^G(\chi)) \cong \mathbb C$.
For this we use a result of Kutzko~\cite{MR0442145} and a few ideas of Brown~\cite{MR0352324}.
Then we show that there exists a subgroup $H'$ and its character $\chi'$ such that $\mathrm{End}_G(\ind_{H'}^G(\chi')) \cong \mathbb C$ and $V_{H'}(\chi')$ is non-trivial. In the proof we see that finite weight condition is used crucially.
Next, we establish a non-trivial $G-$linear map between $\pi$ (given finite weight representation) and $\ind_{H'}^{G}(\chi')$. Then we indicate the proof of the fact that  $\ind_{H'}^{G}(\chi')$ is irreducible for above obtained $H'$ and $\chi'$. Combining this all together, we obtain $\pi \cong \ind_{H'}^{G}(\chi')$ and therefore $\pi$ is monomial.

Our next results are regarding the description of finite dimensional representations of the finitely generated two step nilpotent groups. First of all, we provide a full description of the finite dimensional representations of two step
 	groups whose center has rank one.

Define $\mathbb H(s_1, s_2,
 \ldots, s_n) = \{(u,v,z) \mid z \in \mathbb Z, u,v \in \mathbb Z^n
 \}$ where the group operation is defined by,

 \begin{eqnarray}
     (u_1, u_2, \ldots, u_n, v_1, v_2, \ldots, v_n, z)(u'_1, u'_2, \ldots, u'_n, v'_1,
     v'_2, \ldots, v'_n, z') =  \nonumber\\
     (u_1 + u'_1, \ldots, v_n + v'_n, z+ z' + \sum_{i=1}^n s_i u_i v'_i) \nonumber
 \end{eqnarray}

It is well known that a two step nilpotent group with rank one center is of the form $\mathbb H(s_1, s_2,
 \ldots, s_n)$.

 \begin{thm}
 	\label{rank-one}
 	Let $G = \mathbb H(s_1, s_2,\ldots, s_n)$ and $(\rho, V)$ be an
 	irreducible monomial representation of $G$ with respect to $(H(V), \chi)$. Then the
 	following are equivalent.
 	
 	\begin{enumerate}
 		
 		\item $\chi|_{[G,G]}$ is of order $N_1 < \infty .$
 		
 		\item $V$ is finite dimensional. Further for all $1 \leq j \leq
 		n$, let $N_j$ be the least positive integer such that $\frac{N_j
 			s_j}{s_1}$ is a multiple of $N_1$. Then $\mathrm{dim}(V) = N_1 N_2
 		\ldots N_n$.
 		
 	\end{enumerate}
 	
 \end{thm}
This is a result analogous to Theorem 4 in \cite{MR3201567}.  \\

In the end, we look at finite dimensional representations of
finitely generated two step nilpotent groups more closely and prove the following.
\begin{thm}
	\label{general-case}
	Let $G$ be a finitely generated two step nilpotent group and $\rho$ be an irreducible representation of $G$. Then $\rho$ is finite dimensional if and only if character $\chi$ obtained by restricting $\rho$ to $[G, G]$ has
	finite order. Further, in this case the following hold.
	\begin{enumerate}
		\item order of $\chi$ divides $dim(\rho)$.
		\item $dim(\rho) = \sqrt{|G_\rho / Z(G_\rho)|}$ where $G_\rho = G/Ker (\rho)$.
		
	\end{enumerate}
\end{thm}

\section{Preliminaries}
In this section we recall basic definitions, fix notation and prove few important results that are
required to prove Theorem~\ref{main-theorem}. \\

Let $G$ be a group. An algebraic representation of $G$ is a
homomorphism, $\pi: G \rightarrow \mathrm{Aut}(V)$, from group $G$
to the set of automorphisms of a complex vector space $V$. We remark that $V$ can possibly be infinite dimensional. We
denote this by $(\pi, V)$. We also denote this just by either
$\pi$ or $V$ whenever our meaning is clear from the context. If
$V$ is a finite dimensional vector space, we say $\pi$ is a finite
dimensional representation of $G$ and its dimension is equal to
that of $V$. The induced representation is defined as follows.

\begin{definition}(Induced representation)
\label{defn:induced}
Let $H$ be a subgroup of a group $G$ and $(\rho, W)$ be a
representation of $H$. The induced representation
$(\widetilde{\rho}, \widetilde{W})$ of $\rho$ from $H$ to $G$ has
representation space $\widetilde{W}$ consisting of functions $f: G
\rightarrow \End(W)$ satisfying the following:
\begin{enumerate}
    \item $f(hg) = \rho(h) f(g)$ for all $g \in G$, $h \in H$.
    \item Support of $f$ is contained in a set of finite number of
    right cosets of $H$ in $G$.
\end{enumerate}
Then the homomorphism $\widetilde{\rho}: G \rightarrow
\mathrm{Aut}(\widetilde{W})$ is given by $\widetilde{\rho}(g)f(x)
= f(xg)$ for all $x, g \in G$. We denote this induced
representation by $\ind_H^G(\rho)$.
\end{definition}
In the language of group algebras, induced representation $\widetilde{W}$ of $G$, satisfies $\tilde{W} \cong \mathbb C[G] \otimes_{\mathbb C[H]} W$ with the action of $G$ given by multiplication on the left. We remark that such representations have been already studied in literature (see for example \cite{passman-infinite-group-rings}).

\begin{definition}(Finite weight representation)
	A representation $(\pi, V)$ of a group $G$ is said to have {\bf
		finite weight} if there is a subgroup $H$ and a character $\chi: H
	\rightarrow \mathbb C^{\times}$ such that the space $V_H(\chi)$
	defined by,
	\[
	V_H(\chi) = \{v \in V \mid \pi(h) v = \chi(h) v\,\, \mathrm{ for \,\, all}\,\,  h \in H \},
	\]
	is finite dimensional.
\end{definition}

Let $G$ be a finitely generated group and $H$
be a subgroup of $G$. Let $(\chi, W)$ be a one dimensional representation of $H$. Let
$\chi^{g}$ be the conjugate of $\chi$ acting on $g^{-1} H g$.
By $I(\chi, \chi^{g})$, we mean the
space of $H \cap g^{-1}Hg$-linear maps from $(\chi, W)$ to
$(\chi^{g}, W)$ when both of these are viewed as
representations of the group $H \cap g^{-1}Hg$. The following result regarding the space of intertwining operators of induced representations is a slight generalization of a result that was proved by Kutzko~\cite{MR0442145}. Its unitary analogue was proved by
Mackey~\cite{MR0042420} himself in early 50's.

\begin{thm}\label{thm:kutzko} Let $H$ be a subgroup of a discrete group $G$ and let $\chi$ be  character of $H$. Let $\Lambda$ be the set of $(H,H)$-double coset representatives $g \in G$ such that $HgH$ is union of finite number of right cosets of $H$ in $G$. Then the following are true.
	\begin{enumerate}
		
		\item  The space $\mathrm{End}_G(\ind_H^G(\chi))$
		is isomorphic to the set of functions $s : G \rightarrow \mathrm{End}_{\mathbb C}(W)$ such that $s(h_1gh_2) = \chi(h_1) s(g) \chi(h_2)$ for all $h_1, h_2 \in H$, $g
		\in G$ and the support of $s$ is contained in $\Lambda$.
	
		\item  $\mathrm{End}_G(\ind_H^G(\chi)) \cong \oplus_{g \in \Lambda } I(\chi, \chi^{g})$
		
	\end{enumerate}
	
\end{thm}

\begin{proof} The proof of this follows from Kutzko~\cite[p.2]{MR0442145}. Here we outline ideas of the proof.
	For any $w \in W$, define the set of functions $f^w:G \rightarrow
	W$ by
	\[
	f^w(x) = \begin{cases} \chi(x)w & \text{for} \,\, x \in H \\ 0  & \text{otherwise}. \end{cases}
	\]
	Then
	$f^w \in \ind_H^G(\chi)$. For $\psi \in  \mathrm{End}_G(\ind_H^G(\chi))$, define $ s_\psi: G \rightarrow \mathrm{End}(W)$ by  $s_\psi(g)(w) =
	\psi f^w(g).$
	
	Then $\Psi: \psi \mapsto s_\psi$ gives the required isomorphism
	between $\mathrm{End}_G(\ind_H^G(\chi))$ and $\Delta$. As for the second part, for each $g \in H \backslash G /H $, let $\Delta_{g}$ consists of
	functions in $\Delta$ that have their support on $HgH$. Then
	$\Delta \cong \oplus_{g \in \Lambda} \Delta_{g}$.
	
	For all $g \in G \backslash H /G$, we have
	$\Delta_{g}$ is isomorphic to  $I(\chi, \chi^{g})$ with isomorphism given by $s \mapsto
	s(g)$.
\end{proof}
\begin{remark}
\label{rk:general-intertwiner} In general, when  $\rho$ is any representation of a subgroup $H$ of $G,$ the following result is true by {\it Mackey's formula} (see \cite[Chapter~1, Section~5.5, 5.7]{Vigneras96}, also see \cite{Beloshapka-Gorchinskiy}).
\begin{eqnarray}
\label{eq:mackey-formula}
\ind_H^G(\rho)|_{H} \cong \oplus_{g \in H \backslash G /H } \ind_{H \cap H^g}^H(\rho^g|_{H \cap H^g}).
\end{eqnarray}
This gives the following description of $G$-linear maps of $\ind_H^G(\rho)$.
\begin{eqnarray}
\label{frob-reci}
\mathrm{End}_G(\ind_H^G(\rho)) \cong \oplus_{g \in H \backslash G/  H} \mathrm{Hom}_H(\rho, \ind_{H^g \cap H}^H(\rho^g|_{H^g \cap H})  ).
\end{eqnarray}
Further in case the index of $H$ in $G$ is finite and $\pi$ is an arbitrary representation of $G$, then the following vector spaces are isomorphic (see \cite[Chapter~1, Section~5.4]{Vigneras96}).
\begin{eqnarray}
\label{frobenius-reciprocity}
\mathrm{Hom}_G(\pi, \ind_H^G(\rho)) \cong \mathrm{Hom}_H(\pi|_H, \rho).
\end{eqnarray}

\end{remark}

If the group $G$ is a finitely generated nilpotent group the above proposition can be simplified to give an important condition on $H$ and $\chi:H
\rightarrow \mathbb C^{\times}$ such that $\mathrm{ind}_H^G(\chi)$
is irreducible. Recall that $G$ is called nilpotent if every homomorphic image of $G$ has a cyclic normal subgroup.

\begin{definition}(Radical) A subgroup $K$ of $G$ is called {\bf radical} if
$x^n \in K$ for some $n \in \mathbb N$ implies that $x \in K$ (in
\cite{MR0352324}, Brown calls $K$ isolated).
\end{definition}
Intersection of two
radical subgroups of $G$ is radical. Therefore given a subgroup
$H$ of $G$ smallest radical subgroup of $G$ containing $H$ makes
sense, we shall call this radical of $H$ and denote this by
$\sqrt{H}$. The following lemma lists useful properties of radical
of a subgroup and their normalizers in finitely generated nilpotent groups.

\begin{lemma}\label{properties-of-radical} Let $H$ and $K$ be subgroups of a finitely generated nilpotent group $G$  such that $K \subseteq H$. Then the following are true.

\begin{enumerate}

\item If some odd power of each element of a set of generators of $H$ lies in $K$ then $K$ has finite index in $H$ and $\sqrt{H} = \sqrt{K}$.

\item $\sqrt{(\sqrt{H})} = \sqrt{H}$.
\item $\sqrt{N_G(H)} = N_G(\sqrt{H})$.

\end{enumerate}

\end{lemma}

\begin{proof} Follows from Baumslag~\cite[Lemma~2.8]{MR0283082} and Brown~\cite[Lemma 4]{MR0352324}.

\end{proof}

\begin{lemma}
    \label{lem: radical}
    \label{remark-kutzko}
   Any $g \in G$ such that $HgH$ is union of finite number of right cosets of  $H$ in $G$, has the property that for all $h \in H$ there exists $\kappa(h, g) \in \mathbb N$
    such that $(g h g^{-1})^{\kappa(h, g)} \in H$.
\end{lemma}
\begin{proof}
 Let $g \in G$ be such that $HgH$ is union of finite number of right cosets of $H$ in $G$. Then for any $h \in H$, the sets $Hg h^i$ can not be mutually distinct right cosets of $H$ in $G$. Therefore, there exists $\kappa(h,g)$, depending on $h$ and $g$, such that
	$H gh^{\kappa(h,g)} = Hg$. This implies, $(g hg^{-1})^{\kappa(h, g)} \in H$ for all $h \in H$.


\end{proof}

\begin{proposition}\label{reduction-to-normalizer} If $H$ is a subgroup of $G$,
$\chi$ a character of $H$ and for all $g \in N_G(\sqrt{H})
\setminus H$ we have $\chi^g \neq \chi$ on $H^g \cap H$, then
$\mathrm{Ind}_H^G(\chi)$ is Schur irreducible.
\end{proposition}
\begin{proof} By Theorem~\ref{thm:kutzko}, it is enough to prove that $g \in \Lambda$ implies
that $g \in N_G(\sqrt{H})$. From Lemma~\ref{lem: radical}, we
have $g \in \Lambda $ implies that $g Hg^{-1} \subseteq
\sqrt{H}$. Therefore $g \sqrt{H} g^{-1} \subseteq \sqrt{\sqrt{H}}$
and by Lemma~\ref{properties-of-radical}(2), we get $ g
\sqrt{H} g^{-1} \subseteq \sqrt{H}$, that is $g \in
N_G(\sqrt{(H)})$.

We remark that the equivalence of the one dimensionality of the space
of intertwining operators and irreducibility of representations, though true for finite
groups or unitary representations, need not hold for other cases. This was the gap in our earlier draft which was pointed out to us by Prof. Parshin. Below we give proof of this fact, combining ideas of \cite{Beloshapka-Gorchinskiy} and \cite{Arnal-Parshin}.

\end{proof}
\begin{thm}\label{thm:belo-gorch}Let $G$ be a finitely generated nilpotent group. Let $H$
be a subgroup of $G$ and $W$ be an irreducible representation of $H$. If $\mathrm{End}_G(\mathrm{Ind}_H^G(\chi)) \cong \mathbb C$, then the representation
$\mathrm{Ind}_H^G(W)$ is irreducible.
\end{thm}
\begin{proof}  First of all, we consider the case of $H$ normal subgroup of $G$ and $W$ being one dimensional representation of $H$. We denote the action of $H$ on $W$ by $\chi$ and the induced representation $\mathrm{Ind}_H^G(\chi)$ by $(V, \pi)$. By Remark~\ref{rk:general-intertwiner}, we have $V|_H = \oplus_{g \in G/H }  W^g$. To prove that $V$ is irreducible it is enough to show that any $v \in V$ generates $V$. We note that any $v \in V$ can be written as
\begin{equation}
v = \sum_i v_i, \,\, v_i \in W^{g_i},
\end{equation}
such that only finitely many are non-zero. In case, $v = v_i$ for some $i$ then $v$ clearly generates $V$ because $W^g$ is irreducible representation of $H$. To prove our result, we use induction on number of non-zero constituents of $v$. Let $v = v_1 + v_2 + \cdots + v_k$ such that $v_i \in W^{g_i}$ are all non-zero and $k \geq 2$.

The group $H$ is normal in $G$ implies $HgH$ for any $g \in G$ has finite image in $H \backslash G$. Therefore by Theorem~\ref{thm:kutzko}, we have $\chi^g \neq \chi$ on $H$ for any $g \notin H$.  In other words, $\chi^{g_1} \neq \chi^{g_k}$ on $H$ for $Hg_1 \neq H g_k$. We fix $h \in H$ such that $\chi^{g_1}(h) \neq \chi^{g_k}(h)$. Therefore we have the following,
\begin{equation}
((\chi^{g_k}(h))^{-1}\pi(h)-I)(v) = \sum_{i =1}^{k-1} (\chi^{g_k}(h)^{-1} \chi^{g_i}(h) - 1)v_i \in V.
\end{equation}
By the choice of $h$, we have $((\chi^{g_k}(h))^{-1}\pi(h)-I)(v)$ is non-trivial and has number of non-zero components strictly less than that of $v$.Therefore, the result follows by induction.

Next, we  consider the case when group $H$ is normal and $W$ is of arbitrary dimension.
For this case, just to make it clear, we denote action of $H$ on $W$ by $\rho$ instead of $\chi$. The only place, where the above argument is not applicable is in the definition of $((\rho^{g_k}(h))^{-1} \pi(h)-1)$. Since $\rho(h)$ is not a scalar so we can't use this operator as such. To justify the existence of a parallel operator we use the language of group algebra and modules as given in Passman~\cite{passman-infinite-group-rings}. Let $\mathbb C[H]$ be the infinite group algebra of $H$ over $\mathbb C$. Then for $v_1, v_k \in V$ as above, we have  $\mathbb C[H] v_1$ and $\mathbb C[H] v_k$ are irreducible $\mathbb C[H]$-modules. Let $I^{g_1}$ and $I^{g_k}$ denote the annihilators of $v_1$ and $v_k$ respectively in $\mathbb C[H]$. Then we have the following,
\[
\mathbb C[H]/I^{g_1} \cong \mathbb C[H] v_1 \cong W^{g_1} \ncong W^{g_k} \cong \mathbb C[H] v_k \cong \mathbb C[H]/I^{g_k}.
\]
The ideals $I^{g_1}$ and $I^{g_k}$ are clearly distinct. Therefore, there exists $\alpha \in \mathbb C[H]$ such that $\alpha \in (I^{g_1} \setminus I^{g_2}) \cup (I^{g_2} \setminus I^{g_1})$, $\alpha(v)$ is non-trivial and its number of non-zero components is strictly less than that of $v$. Now onwards, the result follows by induction in this case.

For the general case when $H$ is a subgroup (not necessarily normal) of a finitely generated group $G$ and $W$ is an irreducible representation of $H$ of an arbitrary dimension. We assume that $G$ is nilpotent (note that this hypothesis was not required above). Therefore, there exists a sequence of subgroups $H_i$ for $1 \leq i \leq k$ such that,
\[
H = H_0 \trianglelefteq H_1 \trianglelefteq \cdots \trianglelefteq H_k = G.
\]
We call this a normal series of $H$ of length $k$. We use induction on the minimum length of a normal series of a subgroup of $G$. In case $H$ has a normal series of length one then result follows from above. Let result be true for all subgroups with normal series of length less than or equal to $k-1$ and we prove it for $k$. For this, first we note that $\widetilde{W} = \ind_H^{H_1}(W)$ is Schur irreducible by (\ref{frob-reci}), as $\mathrm{End}_{H_1}(\widetilde{W})$ embeds into $\mathrm{End}_G(\ind_H^G(W)) \cong \mathbb C$. The group $H$ is normal in $H_1$, therefore by above $\widetilde{W}$ is irreducible. Next, we note that the representation $\ind_{H_1}^G(\widetilde{W}) \cong \ind_H^G(W)$ is Schur irreducible and the minimum length of a normal series of $H_1$ is less than $k$, therefore we obtain $\ind_H^G(W)$ is irreducible.

\end{proof}

\begin{lemma}
\label{non-zero-intertwiner}Let $(\pi, V)$ be an irreducible representation of $G$ such that
$V_H(\chi) \neq 0$. Then  $I(\mathrm{Ind}_H^G(\chi), \pi) \neq 0$.

\end{lemma}

\begin{proof} Let $\rho$ denote the representation $\mathrm{Ind}_H^G(\chi).$ We prove that the space of intertwining operators $I(\mathrm{Ind}_H^G(\chi), \pi)$
has positive dimension. Then the result follows from Schur's
lemma. Let $X = \{g_i \mid i \in \mathbb I \}$ be the right coset
representatives of $H$ in $G$. For a fixed $v \in V_H(\chi)$, let
$W$ be the one dimensional space generated by $v$. We define
functions $f_i: X \rightarrow W$ by $f_i(g_j) = \delta_{i,j}(v)$ and
extend these to $G$ so that $f_i \in \mathrm{Ind}_H^G(\chi)$ for
all $i \in \mathbb I$. Then any $f \in \mathrm{Ind}_H^G(\chi)$ can
be written as linear combination of $f_i$'s. For any $g \in G$, we
have

\begin{eqnarray}\label{g-action}
\rho(g)f_i(hg_j) = f_i(hg_jg) = \chi(h) f_i(g_jgg_{i}^{-1}g_i)
\end{eqnarray}

Therefore $\rho(g)f_i$ is nonzero on $g_j$ for $g_j$ satisfying
$g_jg \in H{g_i}$. Thus $\rho(g)f_i = \chi(g_jgg_{i}^{-1})f_j$ for
$j$ such that $g_jg \in H{g_i}$. Now define $F:
\mathrm{Ind}_H^G(\chi) \rightarrow \pi$ by $F(f_i) =
\pi(g_i^{-1})v$ on $f_i$'s and extended linearly thereof. Then,
        \[
        F(\rho(g)f_i) = F(\chi(g_jgg_{i}^{-1}) f_j) = \chi(g_jgg_{i}^{-1})\pi(g_j^{-1})v = \pi(g) \pi(g_i^{-1})v.
        \]
This shows that $F \in I(\mathrm{Ind}_H^G(\chi), \pi)$ is a
non-zero inertwiner and therefore $\mathrm{Ind}_H^G(\chi) \cong
\pi $.

\end{proof}
The above lemma implies the following useful result.
\begin{proposition}\label{reduction}
	
	Let $(\pi, V)$ be an irreducible representation of $G$ such that
	$V_H(\chi) \neq 0$. If $\mathrm{Ind}_H^G(\chi)$ is irreducible
	then $\pi \cong \mathrm{Ind}_H^G(\chi)$. Therefore $\pi$ is
	monomial in this case.
	
\end{proposition}

The following two lemmas whose proofs are fairly standard will play a crucial role in the proof of the main result.

\begin{lemma}\label{diamond-lemma}

Let $H$ and $K$ be two subgroups of $G$ and $\chi: H \rightarrow
\mathbb C^{\times}$, $\delta: K \rightarrow \mathbb C^{\times}$ be
characters of $H$ and $K$ respectively such that,

\begin{enumerate}

\item $kHk^{-1} \subseteq H$ for all $k \in K$, i.e. $K$
normalizes $H$.

\item $\chi(khk^{-1}) = \chi(h)$ for all $h \in H$ and $k \in K$.

\item $\chi|_{H \cap K} = \delta|_{H \cap K}$

\end{enumerate}

Then $\chi \delta: HK \rightarrow \mathbb C^{\times}$ defined by
$\chi \delta(hk) = \chi(h) \delta(k)$ for all $h \in H$ and $k \in
K$ is a character of $HK$ such that $\chi \delta|_{H} = \chi$.

\end{lemma}
\begin{proof}
	The well defined-ness of $\chi \delta$ follows because $h_1 k_1 = h_2 k_2$ implies $h_2^{-1} h_1 = k_2 k_1^{-1} \in H \cap K$. Therefore $\chi(h_2)^{-1} \chi(h_1) = \delta (k_2) \delta(k_1)^{-1}$. Rest of the proof follows easily.

\end{proof}	
\begin{lemma} (see \cite{Hall1950})
	\label{finite-finite.index}
	For a finitely generated group $G$, there exists only finitely many subgroups of a given index.

\end{lemma}

\section{Proof of Theorem~\ref{main-theorem}}

In this section we prove the main theorem.

\begin{proof}
Let $\pi$ be a countable dimensional irreducible
representation of $G$ having finite weight with respect to $(H, \chi)$. We prove the existence of
a subgroup $H'$ and $\chi': H' \rightarrow \mathbb C^{\times}$
such that the following hold
\begin{enumerate}
\item For all $g \in N_G(\sqrt{H'}) \setminus H'$ we have $(\chi')^g \neq \chi'$ on $(H')^g \cap H'$.
\item $V_{H'}(\chi') \neq 0$.
\end{enumerate}
Then by Proposition~\ref{reduction-to-normalizer} and Proposition~\ref{reduction}, we obtain $\pi \cong \ind_{H'}^G(\chi')$ and therefore $\pi$ is monomial. We divide the proof into several steps.
\bigskip

{\bf (a)} We firstly suppose that there exists $g \in N_G(\sqrt{H})
\setminus \sqrt{H}$ such that $\chi^g = \chi$ on $H \cap H^g$.
\bigskip

Then, by the definition of $\sqrt{H}$, the element $g$ has infinite order. Notice that, the groups $H^{g^i}$ all have equal finite index in $\sqrt{H}$. Since the group $\sqrt{H}$ is finitely generated, by Lemma~\ref{finite-finite.index}, we obtain that
\[
H^{g^k} = H \,\, \mathrm{for\,\, some}\,\, k \in \mathbb N,
\]
and
\[
H_0 = \cap_{i=1}^\infty H^{g^i}.
\]
has finite index in
$H$.
Therefore the infinite group generated by $g^k$ acts on irreducible representations of $H$ and $\chi = \chi^{g^{ki}}$ for all $i$ on $H_0$. By Lemma~\ref{non-zero-intertwiner}, we have
 \[
 I( \ind_{H_0}^H(\chi|_{H_0}) , \chi^{g^{ki}}) \neq 0 \,\,\forall \,\, i \geq 0.
 \] But $\ind_{H_0}^H(\chi|_{H_0})$ is finite dimensional and therefore $\chi = \chi^{g^{ki}}$ for some $i$. Also we have $H = H^{g^{ki}}$.
 Therefore there exists a character $\delta$ of group generated by $g^{ki}$, say $K,$ such that
 \[
 V_{H}(\chi) \cap V_{K}(\delta) \neq 0
 \]
 We apply Lemma~\ref{diamond-lemma} for this $H$, $\chi$, $K$, $\delta$ and obtain a character
$\chi_1 $ of group $H_1 = \langle H, g^{ki} \rangle$ such that $H \subsetneq H_1$, and $\chi_1|_{H} = \chi$. It can be easily seen that $V_{H_1}(\chi_1) \neq 0$ is finite dimensional. Thus now onwards we assume that $(H, \chi)$ are chosen so that pair $(H, \chi)$ is such that torsion free rank of $H$ is maximum with the property that $V_H(\chi) \neq 0$ is finite dimensional. From above discussion it also follows that maximal pair $(H, \chi)$ satisfies the following.

(*) The character $\chi$ of $H$ is such that for any $ g \in N_G(\sqrt{H})
\setminus \sqrt{H}$ we have $\chi^g \neq \chi$ on any finite index subgroup of $H$. \\

{\bf (b)} For the next step, if for this maximal $(H, \chi)$ there exists $g \in \sqrt{H} \setminus H$
such that $\chi^g = \chi$ on $H \cap H^g$, we will modify our pair $(H, \chi)$ as follows otherwise $(H, \chi)$ itself satisfies $\pi \cong \ind_H^G(\chi)$.

Let $H_0 = \cap_{g \in N_G(\sqrt{H})} gHg^{-1}$, $\chi_0 = \chi|_{H_0} $ and
\[ L = \{ g \in G \mid \chi_0^g = \chi_0 \,\, \mathrm{on} \,\, H_0^g \cap H_0 \}.
\]
By Proposition~\ref{remark-kutzko}, we have $L \subseteq N_G(\sqrt{H_0}) = N_G(\sqrt{H})$. By definition, the group $H_0$ is normal in $N_G(\sqrt{H})$. By (*), we get that $\chi_0^g \neq \chi_0$ for any $g \in N_G(\sqrt{H_0}) \setminus \sqrt{H_0}$. Hence it suffices to consider  \[
L = \{ g \in
\sqrt{H} = \sqrt{H_0} \mid \chi_0^g= \chi_0 \}.
\]
If $H_0 = L$, then the pair $(H_0, \chi_0)$ is the required pair. If not we have $H_0$ is a proper normal subgroup of $L$. As $L/H_0$ is nilpotent, it has a proper cyclic normal subgroup generated by, say, $gH_0$ for $g \in L \setminus H_0$. Let $S
= \langle g \rangle$ be the cyclic subgroup generated by $g$. Then $H_0S$ is a normal subgroup
of $L$ generated by $H_0$ and $g$. Now $g \in \sqrt{H}$ gives that there exists $t \in \mathbb N$ such that $g^t \in H_0$. As before there exists a character $\delta: S
\rightarrow \mathbb C^{\times}$ such that
\[
V_{H_0}(\chi_0) \cap V_S(\delta) \neq 0
\] Then by
Lemma~\ref{diamond-lemma} for these $H_0$, $\chi_0$, $S$, $\delta$, we get that there exists a character
$\chi_1$ of $H_0 \subsetneq H_1 = H_0S $ that extends $\chi$ and $V_{H_1}(\chi_1) \neq 0$.

Let $L_1 = \{ g \in L \mid (\chi_1)^g = \chi_1\}$. Now
\[
H_0 \triangleleft H_1 \trianglelefteq L_1 \subseteq L \subseteq
\sqrt{H}.
\]
If $L_1 \neq H_1$, then we obtain a group $H_2$ containing $H_1$ as a proper normal subgroup and its character $\chi_2$ that extends $\chi_1$ and therefore satisfying $V_{H_2}(\chi_2) \neq 0$. We note that $H_0$ has finite index in $\sqrt{H}$.
Therefore
continuing this way, there exists a a subgroup $H_k$ of $\sqrt{H}$
and a character $\chi_k$ of $H_k$ such that $H_k = L_k$, $(\chi_k)|_{H_0} = \chi_0$, $V_{H_k}(\chi_k) \neq 0$,
and $(\chi_k)^g \neq \chi_k$ on $H_k^g \cap H_k$ for any $g \in N_G(\sqrt{H_k}) \setminus H_k$. The last assertion follows by choice of $H_k$ and (*). \\



Conversely, Let $\pi \cong \mathrm{Ind}_H^G(\chi)$ be a monomial irreducible representation
of $G$ acting on representation space $V$. We prove that
$V_H(\chi)$ is in fact a one
dimensional subspace of $V$.

Let $\{g_i \mid i \in \mathbb I \}$ be a set of double coset
representatives of $H$ in $G$. For each $i \in \mathbb I $, define
functions $f_i$ such that $f_i(g_j) = \delta_{i,j}$ and then
extended to the whole group $G$ so that $f_i(hg_j) =
\chi(h)f_i(g_j)$ for all $j$. From the definition of
$\mathrm{Ind}_H^G(\chi)$, it is clear that every $f \in V$ can be
written as linear combination of $f_i$'s. The representation $V$
is irreducible, therefore by Schur's lemma and Proposition~\ref{remark-kutzko} for any $g \notin
H$, there exists $h \in H \cap H^g$ such that $\chi(h) \neq
\chi^g(h)$.  Let $f \in V$, then
\[
\pi(h)f(g) = f(ghg^{-1}g) = \chi^g(h) f(g) \neq \chi(h)f(g),
\]
therefore $f \in V_H(\chi)$ if and only if $f$ is non-zero only on
trivial coset representative of $H$ in $G$ and here it is
determined by its value on identity element of $H$. Therefore
$V_H(\chi)$ is one dimensional.
\end{proof}
As a corollary, we obtain the following.
\begin{corollary}
	Every finite dimensional irreducible representation of a nilpotent group $G$ is monomial.
\end{corollary}
A proof of this also appears in \cite[Lemma 1]{MR0352324}.

\section{Proof of Theorem~\ref{rank-one}}
\label{Sec: rank-one}
%
%
%
%
%
%
%

From the definition of $\mathbb H(s_1, s_2, \ldots, s_n),$
it is clear that $[G,G] = \{({\bf 0}, {\bf 0}, s_1z) \mid z \in
\mathbb Z \}$,  where {\bf 0} corresponds to the zero element of
abelian group $\mathbb Z^n$.

Now we prove that $(1)$ implies $(2)$. We have $\chi|_{[G,G]}$ is
of order $N_1$ implies that the set $ K = \{({\bf 0}, {\bf 0},
N_1s_1z) \mid z \in \mathbb Z \}$ is in the kernel of $\rho$.
Hence $\rho$ can be considered as a representation of the group
$G' = G/K$. The result that $V$ is finite dimensional is proved if
we prove existence of an abelian normal subgroup of $G'$ of finite
index.

Consider $A = \{ (u, v, z) \mid u \in \oplus_{i = 1}^n(N_i \mathbb
Z), v \in \mathbb Z^n, z \in \mathbb Z \}$. Then $A' = A/K$ is a
normal abelian subgroup of $G'$ of index $N_1N_2\ldots N_n$. Hence
dimension of $V$ is less than equal to  $N_1N_2\ldots N_n$. Let
$\chi' = \chi|_{[G,G]}$ and as $K$ is contained in the kernel of
$\chi'$ we may consider $\chi'$ as a character of $[G,G]/K$. Let
$\delta$ be a character of $A'$ such that

\begin{enumerate}

\item $\delta|_{[G,G]/K} = \chi'$.

\item $<\rho|_{A'}, \delta> \neq 0$.

\end{enumerate}

Such a character $\delta$ exists because $A'$ is an abelian
subgroup of $G'$.

\begin{lemma}

The representation $\mathrm{Ind}_{A'}^{G'} (\delta)$ is an
irreducible representation of $G'$.

\end{lemma}

\begin{proof} To prove this, we use Proposition~\ref{reduction-to-normalizer}.
We consider the set \[
      \{((\alpha_1, \alpha_2, \ldots, \alpha_n), {\bf 0}, 0) \mid 1 \leq \alpha_i < N_i \,\, \mathrm{for \, all} \,\, i \} \subset G
      \]
as the set of representatives of $G'/A'$. We claim that $\delta^g
\neq \delta$ for all $g \in G'/A'$. This is equivalent to show
that for any $g \in G'/A'$, there exists $a \in A'$ such that
$\delta(gag^{-1}a^{-1}) \neq 1$.

Suppose $g = (\alpha_1, \alpha_2, \ldots, \alpha_n), {\bf 0}, 0)$
is such that $j$ is the smallest integer with $\alpha_j \neq 0$.
We take $a = ({\bf 0}, v, 0) \in G$ where  $v \in \mathbb Z^n$
such that $j$th coordinate is one and all other coordinates are
zero. Then $gag^{-1}a^{-1} = ({\bf 0}, {\bf 0}, \alpha_js_j ) $.
But then by the choice, $\alpha_j s_j < N_j s_j$ and therefore
$\alpha_j s_j$ is not multiple of $N_1$, this implies that
$\delta(gag^{-1}a^{-1}) = \chi_{[G,G]}(gag^{-1}a^{-1}) \neq 1$.

\end{proof}

Therefore $\mathrm{Ind}_{A'}^{G'} (\delta)$ is irreducible. Hence
by Lemma~\ref{reduction}, we obtain that $\rho \cong
\mathrm{Ind}_{A'}^{G'} (\delta)$.  Therefore  dimension of $\rho$
is equal to $N_1 N_2 \cdots N_n$.

To prove that $(2)$ implies $(1)$, let $\dim(V) = N_1 N_2\ldots N_n$ such
that $N_i$ are defined as above. Since $\dim(V) < \infty,$ $\rho$
is monomial. Hence each $\rho(g)$ is a monomial matrix and it
follows that there exists a large positive integer $m$ such that
$\rho(g)^m$ is diagonal for all $g \in G.$ Let $a, b \in G$ be
such that $[a, b] = ({\bf{0}}, {\bf{0}}, s_1) \in [G, G].$ Since
$[a^m, b^m] = [a, b]^{m^2}$ (see proof of Lemma \ref{finite}
below), we have $$I = \rho [a^m, b^m] = \chi [a^m, b^m] = \chi
({\bf{0}}, {\bf{0}}, m^2 s_1).$$ Hence $\chi$ has order less than
or equal to $m^2.$

We prove that order of $\chi$ is $N_1$. Suppose order of $\chi$ is
$d_1$. But then by the first part, we have $\dim(V) = d_1d_2\ldots
d_n$. Here, definition of $d_j$ are similar to that of $N_j$ in
the statement of theorem. Since $N_j a_j = b_j N_1,$ where $a_j,
b_j$ are positive integers, it follows from the definition of
$N_j$ that $N_j$ and $b_j$ do not have common factors and so $N_j$
divide $N_1.$ Similarly, $d_j$ divide $d_1.$ Notice that $\dim(V)
= N_1 N_2 \ldots N_n = d_1 d_2 \ldots d_n.$ So the prime factors
appearing in $N_1$ and $d_1$ will have to be the same. Let $N_1 =
p_1^{e_1}p_2^{e_2}\ldots p_k^{e_k}$ and $d =
p_1^{f_1}p_2^{f_2}\ldots p_k^{f_k}$ with $p_i < p_{i+1}$. Suppose
$j$ is smallest integer such that $e_j \neq f_j$. Let $e_j <
f_j$.Then by definition of $d_i$ and $N_i$, we have the power of
$p_j$ appearing in each $d_i$ is less than or equal to that of
$N_j$ with strict inequality for $d_1$ and $N_1$. But then this
contradicts the fact that $\dim(V) =  N_1 N_2 \ldots N_n = d_1d_2
\ldots d_n$. \\

Closely following the proof above we obtain the following
corollary, analogues to Theorem 4 in \cite{MR3201567}.

\begin{corollary}

Let $G = \mathbb H(1, 1, \ldots, 1),$ and let $V = \ind_{H(V)}^G
(\chi)$ be an irreducible representation of $G$ where $\chi: H(V)
\to \mathbb C^{\times}$ is a character. Let $\chi_C$ be the
restriction of $\chi$ to the center $\{ ({\bf {0}}, {\bf {0}}, z):
z \in \mathbb Z \}.$ Then the following are equivalent.

\bigskip

\begin{enumerate}

\item $V$ is finite dimensional and $\dim V = N^n.$

\item $\chi_C$ has order $N < \infty$ in $\mathbb C^{\times}.$

\item $[G : H(V)] = N^n < \infty .$

\item $H(V) = (N\mathbb Z)^n \times \mathbb Z^n \times \mathbb Z.$

\end{enumerate}
\end{corollary}

\section{Proof of Theorem~\ref{general-case}}
\label{Sec: general}
In this section, we generalize some of the results of the last
section to all finitely generated two step nilpotent
groups.


%

\begin{lemma}\label{finite}
Let $G$ be a finitely generated, two step nilpotent
group such that $[G, G]$ is finite cyclic. Then $G/Z(G)$ is
finite.
\end{lemma}

\begin{proof}
Let $k$ be the order of $[G, G]$ we show that $x^k \in Z(G)$ for
all $x \in G.$ First notice that $[x^\ell, y] = [x, y]^\ell.$ This
is easily proved by induction. For, $$[x^{\ell +1}, y] = x^{\ell
+1} y x^{-\ell -1} y^{-1} = x [x^\ell, y] y x^{-1} y^{-1} =
[x^\ell, y] [x, y].$$ It follows that $[x^k, y] = e$ for all $y
\in G,$ and so $x^k \in Z(G).$ Since $G/Z(G)$ is a finitely
generated abelian group and every element in $G/Z(G)$ has finite
order it follows that $G/Z(G)$ is finite.

\end{proof}

\begin{thm}\label{two-step-general-result}
Let $G$ be a finitely generated two step nilpotent group
$\rho$ be an irreducible finite dimensional representation
(and so monomial) of $G.$ Let $G_\rho = G/ Ker \rho,$ then $dim
\rho = \sqrt {|G_\rho / Z(G_\rho)|}.$
\end{thm}

Proof of this theorem requires some lemmas.

\begin{lemma}\label{structure}
Let $\mathcal{G}$ be a finite abelian group and $f: \mathcal{G}
\times \mathcal{G} \to \mathbb C^{\times}$ be an anti-symmetric
non-degenerate co-cycle ($f \in H^2 (\mathcal{G} , \mathbb
C^{\times})).$ Then $\mathcal{G} = A \times \widehat{A}$ where $A$
is finite abelian and $\widehat{A} $ is the group of characters of
$A.$ Moreover, any maximal abelian normal subgroup of
$\mathcal{G}$ has index $|A|$ in $\mathcal{G}.$
\end{lemma}

For a proof see \cite[Lemma~4.2]{MR2742735}.

\bigskip
Before we state the next lemma we introduce some notation. Let
$\mathbb Z_p$ be the cyclic group of order $p.$ By $\mathbb Z_p^2$
we denote the direct product $\mathbb Z_p \times \mathbb Z_p.$
Also recall that the character group of $\mathbb Z_p$ is
isomorphic to itself.

\begin{lemma}
Under the assumptions in Theorem \ref{two-step-general-result}
$$G_\rho / Z(G_\rho) \cong \mathbb Z_{p_1}^2 \times \mathbb Z_{p_2}^2 \times
\ldots \times \mathbb Z_{p_k}^2.$$
\end{lemma}

\begin{proof}
Let $\chi$ be the character obtained by restricting $\rho$ to the
center $Z(G_\rho).$ Since $\rho$ is faithful, so is $\chi.$ Define
$$f : G_\rho / Z(G_\rho) \times G_\rho /Z(G_\rho) \to
\mathbb C^{\times}$$ by
\begin{equation}\label{cocycle}
f (x Z(G_\rho), y Z(G_\rho)) = \chi [x, y].
\end{equation}

It is easily seen that this is well defined. Since $f(y, x) =
\chi[y, x] = \chi [x, y]^{-1} = f(x, y)^{-1}$ it follows that $f$
is anti-symmetric. If $x$ is not in $Z(G)$ there exists $y$ such
that $xy \neq yx.$ Hence $[x, y] \neq 1,$ and since $\chi$ is
faithful, it follows that $f$ is non-degenerate.  Next, we show
that $[G_\rho, G_\rho]$ is finite cyclic. Since $\rho$ is
monomial, each $\rho(g)$ is a monomial matrix. If $a, b \in
G_\rho,$ $aba^{-1}b^{-1} \in [G_\rho, G_\rho] \subset Z(G_\rho).$
Let $m$ be a large enough integer such that $\rho(g)^m$ is
diagonal for all $g \in G_\rho.$ It follows that $I = \rho[a^m,
b^m] = \chi [a^m, b^m].$ Hence $\chi|_{[G_\rho, G_\rho]}$ has
order less than or equal to $m^2 < \infty.$ Since $\chi|_{[G_\rho,
G_\rho]}$ is faithful, it follows that $[G_\rho, G_\rho]$ is
finite cyclic. Applying Lemma \ref{finite} we have $G_\rho/
Z(G_\rho)$ is finite and is abelian as $G_\rho / Z(G_\rho) \subset
G_\rho / [G_\rho, G_\rho].$ By Lemma \ref{structure} we have
$G_\rho / Z(G_\rho) \cong A \times \widehat{A}$ for some $A$ and
from the structure theorem for abelian groups we have the proof.
\end{proof}

Now we are in a position to complete the proof of Theorem
\ref{two-step-general-result}. Choose $H_\rho$ so that $Z(G_\rho)
\subset H_\rho$ and $H_\rho$ is maximal with respect to the
property
\begin{equation}\label{maximal}
\forall (x, y) \in H_\rho \times H_\rho~~ f(xZ(G_\rho),
yZ(G_\rho)) = 1
\end{equation}
where $f$ is given by \ref{cocycle}.

 For $x, y \in H_\rho,$ $f(x, y) = 1 \implies [H_\rho, H_\rho] \subset Ker \chi = {e}.$ It follows
that $H_\rho$ is abelian and $\chi$ extends as a character of
$H_\rho,$ say $\widetilde{\chi}.$ Moreover, since $Z(G_\rho)
\subset H_\rho$ we have that $H_\rho$ is normal in $G_\rho.$ If
$H$ is any normal abelian group such that $H_\rho \subset H,$ then
$H$ will be equal to $H_\rho$ due to maximality of $H_\rho.$ From
Lemma \ref{structure} we have that $|G_\rho / H_\rho| = p_1 p_2
\ldots p_k.$ Next, we claim that the stabilizer of
$\widetilde{\chi}$ in $G_\rho$ equals $H_\rho.$ Now, if there
exists $g \in G_\rho$ such that $\widetilde{\chi}^g (x) =
\widetilde{\chi}(x)~ \forall x \in H_\rho$ we obtain that
$\widetilde{\chi}([g, x]) = 1$ for all $x \in H_\rho.$ If $g$ is
not in $H_\rho$ this will contradict the maximality. By
Proposition \ref{reduction-to-normalizer} $\ind_{H_\rho}^{G_\rho}
(\widetilde{\chi}) $ is irreducible and by Proposition
\ref{reduction} it is equivalent to $\rho.$ This completes the
proof of Theorem \ref{two-step-general-result}.

\bigskip
Conversely, we have

\begin{thm}
Let $G$ be a finitely generated two step nilpotent group
and $\rho$ an irreducible representation of $G.$  Assume that the
character $\chi$ obtained by restricting $\rho$ to $[G, G]$ has
finite order in $\mathbb C^{\times}.$ Then $\rho$ is finite
dimensional and $dim \rho = \sqrt{|G_\rho / Z(G_\rho)|}$ where
$G_\rho$ is $G/Ker \rho.$
\end{thm}

\begin{proof}
Since $\rho$ is faithful on $G_\rho$ it produces a faithful
character $\chi$ (we use the same notation) when restricted to
$[G_\rho, G_\rho].$ Since $\chi$ has finite order, it follows that
$[G_\rho, G_\rho]$ is finite cyclic. Now, we may proceed as above.

\end{proof}


Now to prove that order of $\chi$ devides dimensiona of $\rho$,
It suffices to consider $G_\rho = G/ Ker \rho.$ Using the previous
notation, we show that $\chi^{p_1 p_2 \ldots p_k} = 1.$ Note that,
if $x \in G_\rho,$ then $x^{p_1 p_2 \ldots p_k} \in Z(G_\rho).$
Hence, if $x, y \in G_\rho,$ $$ \chi^{p_1 p_2 \ldots p_k} [x, y] =
\chi [x, y]^{p_1 p_2 \ldots p_k} = \chi[x^{p_1 p_2 \ldots p_k}, y]
= 1.$$


{\bf Acknowledgment.} This work is supported in part by UGC Centre
for Advanced Studies. The second named author thanks Talia
Fern\'os for useful discussion in the context of this work.

\bibliography{h}{}
\bibliographystyle{siam}

\end{document}